\documentclass[reqno]{amsart}

\usepackage{amsfonts, upgreek}
\usepackage{amssymb, amsopn, latexsym,amsmath,graphics,verbatim,amsthm,enumerate,amscd,tabularx}
\usepackage[all,cmtip]{xy}
\usepackage{mathtools, graphicx}
\usepackage{fancyhdr}
\usepackage[active]{srcltx}
\usepackage[british]{babel}
\usepackage{amsmath,amssymb,amsthm,mathrsfs, url}
\usepackage[T1]{fontenc}
\usepackage{ae}
\usepackage{aecompl}

\theoremstyle{plain}
\newtheorem*{theorem*}{Theorem}
\newtheorem*{conjecture*}{Conjecture}

\newtheorem*{conjectureA*}{Conjecture A*}

\newtheorem{thm}{Theorem}[section]
\newtheorem{theorem}[thm]{Theorem}

\newtheorem{lemma}[thm]{Lemma}
\newtheorem*{lemma*}{Lemma}

\newtheorem{proposition}[thm]{Proposition}
\theoremstyle{definition}

\theoremstyle{remark}

\input cyracc.def \font\tencyr=wncyr10 \def\russe{\tencyr\cyracc}
\def\Sha{\text{\russe{Sh}}}

\newdir{ (}{{}*!/-5pt/@^{(}}

\SelectTips{eu}{10}
\UseTips
\entrymodifiers={+!!<0pt,\fontdimen22\textfont2>}

\DeclareMathOperator{\Gal}{Gal}
\DeclareMathOperator{\rank}{rank}

\DeclareMathOperator{\Selm}{Sel}

\DeclareMathOperator{\img}{img}
\DeclareMathOperator{\coker}{coker}
\DeclareMathOperator{\Hom}{Hom}

\DeclareMathOperator{\res}{res}

\DeclareMathOperator{\cor}{cor}
\DeclareMathOperator{\can}{can}
\DeclareMathOperator{\norm}{norm}

\newcommand{\Q}{{\mathbb{Q}}}
\newcommand{\Z}{{\mathbb{Z}}}

\newcommand{\C}{{\mathbb{C}}}

\newcommand{\Zp}{{\mathbb{Z}_p}}

\newcommand{\Qp}{{\mathbb{Q}_p}}
\newcommand{\ilim}{\mathop{\varprojlim}\limits}
\newcommand{\dlim}{\mathop{\varinjlim}\limits}

\newcommand{\Selinf}{{\Selm_{p^{\infty}}}}

\newcommand{\cE}{\mathcal{E}}

\newcommand{\cO}{\mathcal{O}}

\newcommand{\dotcup}{\ensuremath{\mathaccent\cdot\cup}}

\begin{document}

\title{On the $\Lambda$-cotorsion subgroup of the Selmer group}

\author{Ahmed Matar}

\address{Department of Mathematics\\
         University of Bahrain\\
         P.O. Box 32038\\
         Sukhair, Bahrain}
\email{amatar@uob.edu.bh}

\begin{abstract}
Let $E$ be an elliptic curve defined over a number field $K$ with supersingular reduction at all primes of $K$ above $p$. If $K_{\infty}/K$ is a $\Zp$-extension such that $E(K_{\infty})[p^{\infty}]$ is finite and $H^2(G_S(K_{\infty}), E[p^{\infty}])=0$, then we prove that the $\Lambda$-torsion subgroup of the Pontryagin dual of $\Selinf(E/K_{\infty})$ is pseudo-isomorphic to the Pontryagin dual of the fine Selmer group of $E$ over $K_{\infty}$. This is the Galois-cohomological analog of a flat-cohomological result of Wingberg.
\end{abstract}

\maketitle

\section{Introduction}
If $A$ is a Hausdorff, abelian locally-compact topological group we denote its Pontryagin dual by $A^*$. Let $\Gamma$ be a pro-$p$ group isomorphic to $\Zp$ and let $\Lambda=\Zp[[\Gamma]]$ be the completed group ring. If $A$ is a finitely generated $\Lambda$-module, we let $T_{\Lambda}(A)$ denote its $\Lambda$-torsion submodule. Also we let $\dot{A}$ be the $\Lambda$-module $A$ with the inverse $\Lambda$-action: $\gamma \cdot a=\gamma^{-1}a$ for $a \in A, \gamma \in \Gamma$. We denote $T_{\Lambda}(\dot{A})$ by $\dot{T}_{\Lambda}(A)$.

We now define the $p^{\infty}$-Selmer group and the fine $p^{\infty}$-Selmer group. Assume that $p$ is a prime, $F$ a number field and $E$ is an elliptic curve defined over $F$. Let $S$ be a finite set of primes of $F$ containing all the primes dividing $p$, all the primes where $E$ has bad reduction and all the archimedean primes. We let $F_S$ be the maximal extension of $F$ unramified outside $S$. Suppose now that $L$ is a field with $F \subseteq L \subseteq F_S$. We let $G_S(L)=\Gal(F_S/L)$ and $S_L$ be the set of primes of $L$ above those in $S$. We define the $p^{\infty}$-Selmer group of $E/L$ as

$$\displaystyle 0 \longrightarrow \Selinf(E/L) \longrightarrow H^1(G_S(L), E[p^{\infty}]) \longrightarrow \prod_{v \in S_L} H^1(L_v, E)[p^{\infty}]$$

Also we define the fine $p^{\infty}$-Selmer group of $E/L$ as

$$\displaystyle 0 \longrightarrow R_{p^{\infty}}(E/L) \longrightarrow H^1(G_S(L), E[p^{\infty}]) \longrightarrow \prod_{v \in S_L} H^1(L_v, E[p^{\infty}])$$

The goal of this paper is to prove the following result

\begin{theorem}\label{main_theorem}
Let $K$ be a number field, $E$ an elliptic curve defined over $K$ and $p$ a rational prime such that $E$ has good supersingular reduction at all primes of $K$ above $p$. Let $K_{\infty}/K$ be a $\Zp$-extension such that every prime of $K$ above $p$ ramifies and such that: (i) $E(K_{\infty})[p^{\infty}]$ is finite (ii) $H^2(G_S(K_{\infty}), E[p^{\infty}])=0$. Then there exists a pseudo-isomorphism

$$\dot{T}_{\Lambda}(\Selinf(E/K_{\infty})^*) \sim R_{p^{\infty}}(E/K_{\infty})^*$$
\end{theorem}
Concerning the conditions in the theorem, condition (i) is a mild one (see proposition \ref{finite_points_prop} below) whereas condition (ii) implies that $R_{p^{\infty}}(E/K_{\infty})^*$ is $\Lambda$-torsion (see theorem \ref{fine_Selmer_theorem}).

The theorem shows a nice relationship between the structures of the Selmer and fine Selmer group that is not at all clear exists from the definitions of these groups. The Selmer group in the supersingular case is difficult to deal with mainly due to the lack of a control theorem. The above theorem, we hope, will help us understand the structure of the Selmer group by proving results about the fine Selmer group which is more approachable.

Let $p$ be a fixed odd prime. If $K$ is a number field, we let $K^{cyc}$ be the cyclotomic $\Zp$-extension of $K$ and if $K$ is an imaginary quadratic field, we let $K^{anti}$ be the anticyclotomic $\Zp$-extension of $K$. Coates and Sujatha (\cite{CS1} conjecture A) and the author (\cite{Matar} conjecture B) have conjectured when $K_{\infty}=K^{cyc}$ (respectively $K_{\infty}=K^{anti}$) that $R_{p^{\infty}}(E/K_{\infty})^*$ is a finitely generated $\Zp$-module. This is equivalent to $R_{p^{\infty}}(E/K_{\infty})^*$ being a $\Lambda$-torsion module with $\mu$-invariant zero. Taking into  account theorem \ref{fine_Selmer_theorem}, the above theorem then predicts that $\Selinf(E/K_{\infty})^*$ has $\mu$-invariant zero when $E$ has supersingular reduction at primes of $K$ above $p$.

Using the results of Wuthrich \cite{Wuthrich}, one can in some cases prove results on the structure of $R_{p^{\infty}}(E/K_{\infty})^*$ and hence by the above theorem (if it's conditions are met) give results on $T_{\Lambda}(\Selinf(E/K_{\infty})^*)$. To illustrate this, consider the curve $E=X_0(11): y^2+y=x^3-x^2-10x-20$. This curve has supersingular reduction at $p=29$. Let $\Q^{cyc}$ be the cyclotomic $\Z_{29}$-extension of $\Q$. Wuthrich (\cite{Wuthrich} prop. 9.3) has shown that $R_{p^{\infty}}(E/\Q^{cyc})$ is finite. This together with theorem \ref{fine_Selmer_theorem} guarantees that condition (ii) of theorem \ref{main_theorem} is satisfied. Also by proposition \ref{finite_points_prop} condition (i) is satisfied since for example $E$ does not have CM. It follows from theorem \ref{main_theorem} that $T_{\Lambda}(\Selinf(E/\Q^{cyc})^*)$ is finite.

We give one more example. Let $K=\Q(\sqrt{-7})$ and $K_{\infty}/K$ the anticyclotomic $\Z_{29}$-extension of $K$. By using Wuthrich's work, the author has shown (\cite{Matar} sec 4) that $R_{p^{\infty}}(E/K_{\infty})^*$ is $\Lambda$-torsion with $\mu=0$. Hence by theorem \ref{main_theorem} $\Selinf(E/K_{\infty})^*$ has $\mu=0$.

Wingberg (\cite{Wingberg1} corollary 2.5) has proven a similar result to theorem \ref{main_theorem} stated in terms of flat cohomology rather than Galois cohomology. Although it may appear that the above theorem follows from Wingberg's result, the author has found difficulties in attempting such a deduction in the case when a prime $v$ of $K$ where $E$ has bad reduction splits completely in $K_{\infty}/K$. The following argument illustrates the potential obstacles. Let $E$ and $K$ be as in the theorem and let $\cE$ be the N\'{e}ron model of $E$ over $\cO_K$.

To attempt to deduce the above theorem from Wingberg's result, one would hope to show that $\Selinf(E/K_{\infty})^*$ and $H^1(\cO_{\infty}, \cE[p^{\infty}])^*$ are pseudo-isomorphic (where $\cO_{\infty}$ is the ring of integers of $K_{\infty}$). In hope of showing the existence of such a pseudo-isomorphism, one may use the results of \v{C}esnavi\v{c}ius's paper \cite{Ces1} as they are relevant. Assuming that no prime $v$ of $K$ where $E$ has bad reduction splits completely in $K_{\infty}/K$, the proof of \cite{Ces1} prop. 5.4 together with \cite{Ces1} prop. 2.5 show that the difference between the groups $\Selinf(E/K_n)$ and $H^1(\cO_{K_n}, \cE[p^{\infty}])$ is finite and bounded with $n$. This proves that $\Selinf(E/K_{\infty})^*$ and $H^1(\cO_{\infty}, \cE[p^{\infty}])^*$ are pseudo-isomorphic in this case. However in the case when a prime $v$ of $K$ where $E$ has bad reduction splits completely in $K_{\infty}/K$, this argument can fail and hence it is unclear that a pseudo-isomorphism exists in this case.

Also in order to invoke Wingberg's theorem, one needs the $\Lambda$-module $H^2(\cO_{\infty}, \cE[p^{\infty}])^*$ to be torsion. If no prime $v$ of $K$ where $E$ has bad reduction splits completely in $K_{\infty}/K$, then assuming $H^2(G_S(K_{\infty}), E[p^{\infty}])=0$ one can deduce the fact that $H^2(\cO_{\infty}, \cE[p^{\infty}])^*$ is $\Lambda$-torsion from \cite{Gb_IPR} prop. 3, prop. \ref{cohomology_group_vanishing_prop} below, \cite{Schneider2} sec3 corollary 5 and \v{C}esnavi\v{c}ius's results referred to above.  However such a deduction can fail when a prime $v$ of $K$ where $E$ has bad reduction splits completely in $K_{\infty}/K$.

The above arguments illustrate the difficulties in attempting to deduce theorem \ref{main_theorem} from Wingberg's result. Everything is done in this paper with Galois cohomology. Our method of proof generally follows Wingberg's with major differences being that all exact sequences arising from the spectral sequences of Schneider \cite{Schneider} are replaced with sequences arising from the snake lemma together with the Kummer sequence. The other difference is that the Artin-Mazur duality of flat cohomology groups is replaced with the Poitou-Tate duality of Galois cohomology groups.

The following proposition shows that condition (i) in theorem \ref{main_theorem} is a mild one. As the proposition shows, all elliptic curves without complex multiplication satisfy condition (i) in the theorem. For elliptic curves with complex multiplication a slightly weaker version of theorem \ref{main_theorem} is given in \cite{Billot}.

\begin{proposition}\label{finite_points_prop}
With the setup and conditions in theorem \ref{main_theorem}, we have that $E(K_{\infty})[p^{\infty}]$ is finite in the following cases:
\begin{enumerate}
\item $E$ does not have complex multiplication
\item $K_{\infty}/K$ is the cyclotomic $\Zp$-extension of $K$
\item $p$ is odd and splits in $K/\Q$
\end{enumerate}
\end{proposition}
\begin{proof}
(i) Suppose that $E$ does not have complex multiplication. By a theorem of Serre \cite{Serre} this implies that $\Gal(K(E[p^{\infty}])/K)$ is an open subgroup of $GL_2(\Zp)$. Suppose that $E(K_{\infty})[p^{\infty}]$ is infinite. Then either $E(K_{\infty})[p^{\infty}]$ has $\Zp$-corank one or $E[p^{\infty}]$ is rational over $K_{\infty}$. In the first case $V_p(E)=T_p(E) \otimes_{\Zp} \Qp$ has a one-dimensional $\Gal(\bar{K}/K)$-invariant subspace. This clearly contradicts Serre's theorem. In the second case $K(E[p^{\infty}])/K$ is a subextension of $K_{\infty}/K$ and hence must be an abelian extension. This also contradicts Serre's theorem.\\
(ii) Follows from Ribet's theorem \cite{Ribet}\\
(iii) Suppose thet $p$ is odd and splits in $K/\Q$. Choose a prime $v$ of $K$ above $p$. Since $E$ has supersingular reduction at $v$, we have $E(K_v)[p^{\infty}]=E(\Qp)[p^{\infty}]=\hat{E}(p\Zp)[p^{\infty}]$ where $\hat{E}$ is the formal group of $E/\Qp$. By \cite{Silverman} ch. 4 th. 6.1 $\hat{E}(p\Zp)$ has no $p$-torsion if $p$ is odd so $E(K_v)[p^{\infty}]=\{0\}$. Therefore $E(K_{\infty})[p^{\infty}]^{\Gamma}=E(K)[p^{\infty}]=\{0\}$ which implies that $E(K_{\infty})[p^{\infty}]=\{0\}$.
\end{proof}

\section{Proof of Theorem}
Theorem \ref{main_theorem} will be proven in this section. The proof will be broken up into a number of propositions. We keep all the definitions and notation from the introduction and furthermore denote $\Gamma^{p^n}$ by $\Gamma_n$.

Let $A$ be a finitely generated $\Lambda$-module. We let $T_{\Lambda}(A)$ and $T_{\mu}(A)$ be the $\Lambda$-torsion submodule and $\Zp$-torsion submodule of $A$ respectively. Then define $T_{\lambda}(A):=T_{\Lambda}(A)/T_{\mu}(A)$. As in the introduction, we use the notation $\dot{T}_-(A)=T_-(\dot{A})$. We have the following lemma of Wingberg (\cite{Wingberg1} lemma 1.1)

\begin{lemma}\label{inverse_limit_lemma}
Let $A$ be a finitely generated $\Lambda$-module. Then we have pseudo-isomorphisms
\begin{enumerate}[(i)]
\item $\ilim_{n,m}(A^*/p^m)^{\Gamma_n} \sim \dot{T}_{\mu}(A)$
\item $\ilim_{n,m}(A^*[p^m])_{\Gamma_n} \sim \dot{T}_{\lambda}(A)$
\item $\ilim_{n,m}(A^*/p^m)_{\Gamma_n} \sim 0$
\end{enumerate}
where the inverse limits are taken with respect to multiplication-by-$p$ resp. canonical surjection and the norm map resp. canonical surjection.
\end{lemma}

Now let $F$ be a number field, $S$ a finite set of primes of $F$ and $B$ a finite $G_S$-module whose order is only divisible by rational primes lying below primes in $S$. Define $B':=\Hom(B, \mu)$ where $\mu$ is the group of all roots of unity in $\C$. We let $F_S$ be the maximal extension of $F$ unramified outside $S$. Suppose now that $L$ is a number field with $F \subseteq L \subseteq F_S$. We let $G_S(L)=\Gal(F_S/L)$ and $S_L$ be the set of primes of $L$ above those in $S$. Then we have the following perfect Poitou-Tate duality pairing (\cite{NSW} theorem 8.6.7)

\begin{equation}\label{pairing1}
\Sha^1(G_S(L), B') \times \Sha^2(G_S(L), B) \to \Q/\Z
\end{equation}\\
where $\Sha^i(G_S(L), M)$ ($M$ is any $G_S$-module) is defined to be the kernel of the restriction map $H^i(G_S(L), M) \to \prod_{v \in S_L} H^i(L_v, M)$. If $L_{\infty}/F$ is an infinite extension contained in $F_S$ we define $\Sha^i(G_S(L_{\infty}), M) = \dlim \Sha^i(G_S(L'), M)$ where the direct limit is taken over all intermediate finite extensions $L'/L$ contained in $L_{\infty}$ with respect to the restriction maps.

Now if $E$ is an elliptic curve defined over $F$, $p$ a rational prime and $S$ a finite set of primes of $F$ containing all primes dividing $p$, then for any $n \geq 0$ the Weil pairing together with the above pairing give a perfect pairing

\begin{equation}\label{pairing2}
\langle \;,\; \rangle: \Sha^1(G_S(L), E[p^n]) \times \Sha^2(G_S(L), E[p^n]) \to \Qp/\Zp
\end{equation}\\
Now let $L'$ be a finite extension of $L$ contained in $F_S$. The definition of this pairing (see \cite{NSW} theorem 8.6.7) shows that it is induced by the cup product. Therefore for $a \in \Sha^1(G_S(L'), E[p^n])$ and $b \in \Sha^2(G_S(L), E[p^n])$ we have $\langle \cor a, b \rangle = \langle a, \res b \rangle$ where $\cor: \Sha^1(G_S(L'), E[p^n]) \to \Sha^1(G_S(L), E[p^n])$ is the corestriction map and $\res: \Sha^2(G_S(L), E[p^n]) \to \Sha^2(G_S(L'), E[p^n])$ is the restriction map.

The following theorem is well-known (see for example \cite{PR} prop. 1.3.2). Using the above pairing and a control theorem, we will present another proof of this theorem

\begin{theorem}\label{fine_Selmer_theorem}
Let $K$ be a number field, $p$ a rational prime, $K_{\infty}/K$ a $\Zp$-extension and $E$ an elliptic curve defined over $K$. Let $S$ be a finite set of primes of $K$ containing all the primes dividing $p$, all the primes where $E$ has bad reduction and all the archimedean primes. Then $R_{p^{\infty}}(E/K_{\infty})^*$ is $\Lambda$-torsion if and only if $\Sha^2(G_S(K_{\infty}), E[p^{\infty}])^*$ is $\Lambda$-torsion. If $p$ is odd, this statement is equivalent to $H^2(G_S(K_{\infty}), E[p^{\infty}])=0$.
\end{theorem}
\begin{proof}
Suppose that $p$ is odd. According to \cite{Gb_IPR} prop. 4, $H^2(G_S(K_{\infty}), E[p^{\infty}])$ is a cofree $\Lambda$-module. We will now show that $\Sha^2(G_S(K_{\infty}), E[p^{\infty}])=H^2(G_S(K_{\infty}), E[p^{\infty}])$. This together with the result just mentioned will show that the second statement of the theorem will follow from the first. Let $w$ be a prime of $K_{\infty}$ above a prime $v$ is $S$. We will show that $H^2(K_{\infty,w}, E[p^{\infty}])=0$. This is true for archimedean primes $w$ since $p$ is odd.

Now assume that $w$ is nonarchimedean. We consider two cases. First we consider the case where $v$ splits completely in $K_{\infty}/K$. In this case we have $H^2(K_{\infty,w}, E[p^{\infty}])=H^2(K_v, E[p^{\infty}])$. By Tate local duality this group is dual to $T_p(E(K_v)[p^{\infty}])$ (the $p$-adic Tate module of $E(K_v)[p^{\infty}]$) and hence $H^2(K_{\infty,w}, E[p^{\infty}])=0$ since $E(K_v)[p^{\infty}]$ is finite. Now consider the case where $v$ does not split completely in $K_{\infty}/K$. In this case, the extension $K_{\infty,w}/K_v$ is an infinite pro-$p$ extension. Hence by \cite{NSW} theorem 7.1.8(i) $cd_p(K_{\infty, w}) \le 1$. So $H^2(K_{\infty,w}, E[p^{\infty}])=0$ in this case also. This proves that $\Sha^2(G_s(K_{\infty}), E[p^{\infty}])=H^2(G_S(K_{\infty}), E[p^{\infty}])$ as desired.

Now we prove the first statement. By the restriction-corestriction property of the pairing (\ref{pairing2}), the Pontryagin dual of $\Sha^2(G_S(K_{\infty}), E[p^{\infty}])$ can be identified with $\ilim_{n,m} \Sha^1(G_S(K_n), E[p^m])$ where $K_n$ is the fixed field of $\Gamma_n$ and the inverse limit is taken over $m$ with regards to multiplication-by-$p$ and over $n$ with regards to corestriction. Therefore we see that to prove the first statement, we must show that $\rank_{\Lambda}(\ilim_{n,m} \Sha^1(G_S(K_n), E[p^m]))=\rank_{\Lambda}(R_{p^{\infty}}(E/K_{\infty})^*)$.

Consider the group $\ilim_{n,m} \Sha^1(G_S(K_{\infty}), E[p^{\infty}])[p^m]^{\Gamma_n}=\ilim_{n,m} R_{p^{\infty}}(E/K_{\infty})[p^m]^{\Gamma_n}$ where the inverse limit is taken over $m$ with regards to multiplication-by-$p$ and over $n$ with regards to the norm map. According to \cite{NSW} prop. 5.5.10(i) this group is a free $\Lambda$-module with rank equal to the the $\Lambda$-corank of $R_{p^{\infty}}(E/K_{\infty})$. Therefore it will suffice to show that
\begin{align}
\rank_{\Lambda}(\ilim_{n,m} \Sha^1(G_S(K_n), E[p^m]))&=\rank_{\Lambda}(\ilim_{n,m} \Sha^1(G_S(K_{\infty}), E[p^m])^{\Gamma_n}) \label{firstrankequality} \\ &=\rank_{\Lambda}(\ilim_{n,m} \Sha^1(G_S(K_{\infty}), E[p^{\infty}])[p^m]^{\Gamma_n}) \label{secondrankequality}
\end{align}

We first show the equality (\ref{secondrankequality}). For any $m \geq 0$ the snake lemma gives a long exact sequence which we split into two short exact sequences below
\begin{equation}\label{FSG_sequence1}
0 \to \mathcal{M}_{p^m}(E/K_{\infty}) \to \Sha^1(G_S(K_{\infty}), E[p^m]) \to \mathcal{D}_{p^m}(E/K_{\infty}) \to 0
\end{equation}

\begin{equation}\label{FSG_sequence2}
0 \to \mathcal{D}_{p^m}(E/K_{\infty}) \to \Sha^1(G_S(K_{\infty}), E[p^{\infty}])[p^m] \to \mathcal{C}_{p^m}(E/K_{\infty})
\end{equation}\\
where $\mathcal{M}_{p^m}(E/K_{\infty})=E(K_{\infty})[p^{\infty}]/p^m \cap \Sha^1(G_S(K_{\infty}), E[p^m]), \mathcal{D}_{p^m}(E/K_{\infty})=\img(\Sha^1(G_S(K_{\infty}), E[p^m]) \to \Sha^1(G_S(K_{\infty}), E[p^{\infty}])[p^m])$ and $\mathcal{C}_{p^m}(E/K_{\infty})=\coker(E(K_{\infty})[p^{\infty}]/p^m \to \underset{w \in S_{\infty}}{\bigoplus} E(K_{\infty,w})[p^{\infty}]/p^m)$. $S_{\infty}$ being the primes of $K_{\infty}$ above those in $S$.

For any $n \ge 0$, the sequence (\ref{FSG_sequence1}) induces another sequence
$$0 \to \mathcal{M}_{p^m}(E/K_{\infty})^{\Gamma_n} \to \Sha^1(G_S(K_{\infty}), E[p^m])^{\Gamma_n} \to \mathcal{D}_{p^m}(E/K_{\infty})^{\Gamma_n} \to \mathcal{M}_{p^m}(E/K_{\infty})_{\Gamma_n}$$

We claim that all the groups in this exact sequence are finite. Since $E(K_{\infty})[p^{\infty}]/p^m$ is finite, therefore the first and last terms of the sequence are finite. So we only have to show that the third term is finite. This will follow if we show that $\Sha^1(G_S(K_{\infty}), E[p^{\infty}])[p^m]^{\Gamma_n}$ is finite. The finiteness of this group is easily seen by taking Pontryagin duals and noting that $\Sha^1(G_S(K_{\infty}), E[p^{\infty}])^*$ is a finitely generated $\Lambda$-module ($\Sha^1(G_S(K_{\infty}), E[p^{\infty}]) \subseteq \Selinf(E/K_{\infty})$ and $\Selinf(E/K_{\infty})^*$ is a finitely generated $\Lambda$-module by \cite{Manin} theorem 4.5). Therefore we have seen that all the groups in the above exact sequence are finite and so by taking inverse limits the sequence remains exact
\begin{align*}
0 &\to \ilim_{n,m}\mathcal{M}_{p^m}(E/K_{\infty})^{\Gamma_n} \to \ilim_{n,m}\Sha^1(G_S(K_{\infty}, E[p^m])^{\Gamma_n}\\
&\to \ilim_{n,m} \mathcal{D}_{p^m}(E/K_{\infty})^{\Gamma_n} \to \ilim_{n,m}\mathcal{M}_{p^m}(E/K_{\infty})_{\Gamma_n}
\end{align*}

The groups $E(K_{\infty})[p^{\infty}]/p^m$ are finite of bounded order as $m$ varies whence the groups $\mathcal{M}_{p^m}(E/K_{\infty})^{\Gamma_n}$ and $\mathcal{M}_{p^m}(E/K_{\infty})_{\Gamma_n}$ are finite of bounded order as $n$ and $m$ vary. It follows that the first and last inverse limits in the above sequence are finite.

Therefore the map $$\ilim_{n,m} \Sha^1(G_S(K_{\infty}), E[p^m])^{\Gamma_n} \to \ilim_{n,m} \mathcal{D}_{p^m}(E/K_{\infty})^{\Gamma_n}$$ has finite kernel and cokernel which shows that
$$\rank_{\Lambda}(\ilim_{n,m} \Sha^1(G_S(K_{\infty}), E[p^m])^{\Gamma_n})=\rank_{\Lambda}(\ilim_{n,m} \mathcal{D}_{p^m}(E/K_{\infty})^{\Gamma_n})$$
From this and the sequence (\ref{FSG_sequence2}), we see that in order to show the equality (\ref{secondrankequality}) we only have to show that $\ilim_{n,m} \mathcal{C}_{p^m}(E/K_{\infty})^{\Gamma_n}$ is $\Lambda$-torsion. That the group $\ilim_{n,m} \mathcal{C}_{p^m}(E/K_{\infty})^{\Gamma_n}$ is $\Lambda$-torsion follows from the facts that $\ilim_{n,m}(\underset{w \in S_{\infty}}{\bigoplus} E(K_{\infty,w})[p^{\infty}]/p^m)^{\Gamma_n}$ and $\ilim_{n,m}(E(K_{\infty})[p^{\infty}]/p^m)_{\Gamma_n}$ are $\Lambda$-torsion from lemma \ref{inverse_limit_lemma}. Therefore we have established the equality (\ref{secondrankequality}).

We now prove the equality (\ref{firstrankequality}) by means of a control theorem. We denote $\ilim_{n,m} \Sha^1(G_S(K_n), E[p^m])$ by $X_{p^{\infty}}(E/K_{\infty})$ and $\ilim_{n,m} \Sha^1(G_S(K_{\infty}), E[p^m])^{\Gamma_n}$ by $Y_{p^{\infty}}(E/K_{\infty})$. In order to prove the equality (\ref{firstrankequality}), it suffices to show that the map $\Xi: X_{p^{\infty}}(E/K_{\infty}) \to Y_{p^{\infty}}(E/K_{\infty})$ induced by restriction has $\Lambda$-torsion kernel and cokernel. We do this by means of a control theorem. Consider the following commutative diagram

\begin{equation}
\xymatrix{
0 \ar[r] & \Sha^1(G_S(K_{\infty}), E[p^m])^{\Gamma_n} \ar[r] & H^1(G_S(K_{\infty}), E[p^m])^{\Gamma_n} \ar[r]^-{\psi_{{\infty},m}}
& \underset{w \in S_{\infty}}{\bigoplus} H^1(K_{\infty,w}, E[p^m])^{\Gamma_n} \\
0 \ar[r] & \Sha^1(G_S(K_n), E[p^m]) \ar[u]_{s_{n,m}} \ar[r] & H^1(G_S(K_n), E[p^m]) \ar[u]_{h_{n,m}} \ar[r]^-{\psi_{n,m}}
& \underset{v \in S_n} {\bigoplus} H^1(K_{n,v}, E[p^m]) \ar[u]_{g_{n,m}}
}
\end{equation}
In the commutative diagram above the sets $S_n$ and $S_{\infty}$ are the sets of primes above $S$ in $K_n$ and $K_{\infty}$ respectively and the vertical maps are restriction. Taking inverse limits over $n$ and $m$ in the above, we get another commutative diagram

\begin{equation}
\xymatrix{
0 \ar[r] & Y_{p^{\infty}}(E/K_{\infty}) \ar[r] & \ilim_{n,m} H^1(G_S(K_{\infty}), E[p^m])^{\Gamma_n} \ar[r]^-{\upphi}
& \ilim_{n,m} \underset{w \in S_{\infty}} \bigoplus H^1(K_{\infty,w}, E[p^m])^{\Gamma_n} \\
0 \ar[r] & X_{p^{\infty}}(E/K_{\infty}) \ar[u]_{\Xi} \ar[r] & \ilim_{n,m} H^1(G_S(K_n), E[p^m]) \ar[u]_{\Xi'} \ar[r]^-{\uppsi}
& \ilim_{n,m} \underset{v \in S_n} \bigoplus H^1(K_{n,v}, E[p^m]) \ar[u]_{\Xi''}
}
\end{equation}
From the snake lemma, we see that in order to show that $\coker \Xi$ is $\Lambda$-torsion, we only have to show that both $\ker \Xi''$ and $\coker \Xi'$ are $\Lambda$-torsion. Since $cd_p(\Gamma)=1$, therefore it follows that $\coker \Xi'=0$. Now we deal with $\ker \Xi''$. Primes in $S$ that split completely in $K_{\infty}/K$ do not contribute anything to $\ker \Xi''$ so we may assume that $S$ has no such primes.

Now choose an $M$ such that $\#S_M=\#S_{\infty}$ and let $m=\#S_M$. For every $n \geq M$ we label the primes in $S_n$ as $v_1, v_2, ..., v_m$ and the primes of $S_{\infty}$ as $w_1, w_2, ..., w_m$. We choose a labelling such that if $k \geq j \geq M$ then $w_i \in S_{\infty}$ lies above $v_i \in S_k$ lies above $v_i \in S_j$. With this labelling we have
$$\ker \Xi''= \bigoplus_{i=1}^m \ilim_m \ilim_{n \geq M} H^1(\Gal(K_{\infty, w_i}/K_{n, v_i}), E(K_{\infty, w_i})[p^m])$$
where the inverse limit is taken over $n$ with respect to the corestriction maps and over $m$ with respect to multiplication-by-$p$.

For any $n \geq M$ and any $i$ we have $\Gal(K_{\infty, w_i}/K_{n, v_i})=\Gamma_n$, therefore if $g$ is a topological generator of $\Gamma$ we have $H^1(\Gal(K_{\infty, w_i}/K_{n, v_i}), E(K_{\infty, w_i})[p^m])=E(K_{\infty, w_i})[p^m]/(g^{p^n}-1)E(K_{\infty, w_i})[p^m]$. For sufficiently large $n$ we have $E(K_{\infty, w_i})[p^m]=E(K_{n, v_i})[p^m]$, so $(g^{p^n}-1)E(K_{\infty, w_i})[p^m]=\{0\}$ i.e. $H^1(\Gal(K_{\infty, w_i}/K_{n, v_i}), E(K_{\infty, w_i})[p^m])=E(K_{\infty, w_i})[p^m]$. For such sufficiently large $n' \geq n \geq M$ one can check that the corestriction map from $H^1(\Gal(K_{\infty, w_i}/K_{n', v_i}), E(K_{\infty, w_i})[p^m])$ to $H^1(\Gal(K_{\infty, w_i}/K_{n, v_i}), E(K_{\infty, w_i})[p^m])$ is the identity map on $E(K_{\infty, w_i})[p^m]$. This shows that $\ker \Xi''=\bigoplus_{i=1}^m T_p(E(K_{\infty, w_i}))$ (where $T_p(E(K_{\infty, w_i}))$ means the Tate module of $E(K_{\infty, w_i})$). It follows that $\ker \Xi''$ is $\Lambda$-torsion as desired. A similar proof shows that $\ker \Xi'$ is $\Lambda$-torsion, whence $\ker \Xi$ is $\Lambda$-torsion. This completes the proof of the equality (\ref{firstrankequality}) thereby finally finishing the proof of the theorem.
\end{proof}

Throughout the rest of the section let $K$ be a number field, $E$ an elliptic curve defined over $K$ and $p$ a rational prime such that $E$ has good supersingular reduction at all primes of $K$ above $p$. Let $K_{\infty}/K$ be a $\Zp$-extension such that every prime of $K$ above $p$ ramifies. We will assume that (i) $E(K_{\infty})[p^{\infty}]$ is finite  and (ii) $H^2(G_S(K_{\infty}), E[p^{\infty}])=\{0\}$. Finally we let $S_n$ and $S_{\infty}$ be the set of primes of $K_n$ and $K_{\infty}$ above the primes in $S$, respectively.

The first key result is

\begin{proposition}\label{cohomology_group_vanishing_prop}
For any prime $w$ of $K_{\infty}$ above $p$ we have $H^1(K_{\infty, w}, E)[p^{\infty}]=0$
\end{proposition}
\begin{proof}
Let $v$ be the prime of $S$ below $w$. Since by assumption $v$ ramifies in $K_{\infty}/K$ therefore the extension $K_{\infty, w}/K_v$ is deeply ramified in the sense of \cite{CG}. Therefore the result follows as explained in \cite{Gb_LNM} pg. 70.
\end{proof}

Now we need
\begin{proposition}\label{surjectivity_prop}
The map $H^1(G_S(K_{\infty}), E[p^{\infty}]) \xrightarrow{\psi_{\infty}} \underset {w \in S_{\infty}} \bigoplus H^1(K_{\infty,w}, E)[p^{\infty}]$ is surjective
\end{proposition}
\begin{proof}
To understand $\coker \psi_{\infty}$ we use the Cassels-Poitou-Tate exact sequence \cite{CS2}. First for any $n$ and $m$ we define $\Selm_{p^m}(E/K_n)$ as
$$\displaystyle 0 \longrightarrow \Selm_{p^m}(E/K_n) \longrightarrow H^1(G_S(K_n), E[p^m]) \longrightarrow \prod_{v \in S_n} H^1(K_{n,v}, E)[p^m]$$

Then for any $n \geq 0$ the Cassels-Poitou-Tate exact sequence is
$$H^1(G_S(K_n), E[p^{\infty}]) \xrightarrow{\psi_n} \underset {v \in S_n} \bigoplus H^1(K_{n,v}, E)[p^{\infty}] \rightarrow \mathfrak{S}(E/K_n)^* \rightarrow H^2(G_S(K_n), E[p^{\infty}])$$
where $\mathfrak{C}(E/K_n) = \ilim_m \Selm_{p^m}(E/K_n) \subseteq H^1(G_S(K_n), T_p(E))$ (inverse limit with respect to multiplication-by-$p$).

Taking the direct limit of the above sequence with respect to restriction over $n$ we get

$$H^1(G_S(K_{\infty}), E[p^{\infty}]) \xrightarrow{\psi_{\infty}} \underset {w \in S_{\infty}} \bigoplus H^1(K_{\infty,w}, E)[p^{\infty}] \rightarrow \mathfrak{S}(E/K_{\infty})^* \rightarrow H^2(G_S(K_{\infty}), E[p^{\infty}])$$
where $\mathfrak{S}(E/K_{\infty}) = \ilim_{n,m} \Selm_{p^m}(E/K_n) \subseteq \ilim_n H^1(G_S(K_n), T_p(E))$ (inverse limit over $n$ with regards to corestriction and over $m$ with regards to multiplication-by-$p$).

By assumption, we have $H^2(G_S(K_{\infty}), E[p^{\infty}])=\{0\}$ and so we see from the above sequence that $\coker \psi_{\infty}$ is isomorphic to $\mathfrak{S}(E/K_{\infty})^*$. We will show that $\coker \psi_{\infty}=0$ by showing that the Pontryagin dual of $\coker \psi_{\infty}$ is $\Lambda$-torsion while $\mathfrak{S}(E/K_{\infty})$ is $\Lambda$-torsion-free.

First we show that $\coker \psi_{\infty}$ is $\Lambda$-cotorsion. We will also prove that it is cofinitely generated over $\Lambda$ since we will need this fact later. We do this by actually showing that $J:=\underset {w \in S_{\infty}} \bigoplus H^1(K_{\infty,w}, E)[p^{\infty}]$ is cofinitely generated cotorsion over $\Lambda$. Note that by proposition \ref{cohomology_group_vanishing_prop} we may (and will) assume that $S_{\infty}$ contains no primes above $p$. We will also assume that $S_{\infty}$ contains no complex archimedean primes since they contribute nothing to the group $J$. Write $S_{\infty} = T \, \dotcup \, T'$  where $T$ is the set of all the primes of $K_{\infty}$ above those primes of $S$ that do not split completely in $K_{\infty}/K$ and $T'$ is its complement containing all primes of $S_{\infty}$ that lie above a prime of $K$ that splits completely in $K_{\infty}/K$. Let $J_T:=\underset {w \in T} \bigoplus H^1(K_{\infty,w}, E)[p^{\infty}]$ and $J_{T'}:=\underset {w \in T'} \bigoplus H^1(K_{\infty,w}, E)[p^{\infty}]$ so that $J=J_T \times J_{T'}$. For any $w \in T$ by \cite{Gb_IPR} prop. 2 $H^1(K_{\infty,w}, E[p^{\infty}])$ is a cofinitely generated $\Zp$-module and hence the same is true for $H^1(K_{\infty, w}, E)[p^{\infty}]$ as this group is a quotient of $H^1(K_{\infty, w}, E[p^{\infty}])$. Therefore $J_T$ is a cofinitely generated $\Zp$-module.

Now we deal with $J_{T'}$. Let $S'$ be the set of primes of $K$ that split completely in $K_{\infty}/K$. For any such prime $v \in S'$ we let $J_v:=\underset {w |v} \bigoplus H^1(K_{\infty,w}, E)[p^{\infty}]$ where the sum runs over all primes of $K_{\infty}$ above $v$. Clearly $J_{T'}=\underset {v \in S'} \bigoplus J_v$. By Shapiro's lemma, for any $v \in S'$, we have $J_v^{\Gamma}=H^1(K_v, E)[p^{\infty}]$. If $v$ is archimedean, then $H^1(K_v, E)$ is finite (see \cite{GH} prop. 1.3). If $v$ is non-archimedean, then by Tate duality for abelian varieties over local fields (\cite{Milne} I-3.4): $H^1(K_v, E)^* \cong E(K_v)$ so $H^1(K_v, E)[p^{\infty}]^* \cong \ilim E(K_v)/p^m$.

By Mattuck's theorem $E(K_v)=\Z_l^{[K_v:\Q_l]}\times T$ where $l \neq p$ is the characteristic of the residue field of $K_v$ and $T$ is a finite group. It follows that $\ilim E(K_v)/p^m$ is the finite $p$-primary subgroup of $E(K_v)$. So $H^1(K_v, E)$ is finite in the non-archimedean case also. This proves that $J_v^{\Gamma}=H^1(K_v, E)[p^{\infty}]$ is finite which shows that $J_{T'}$ is cofinitely generated over $\Lambda$. Also for any $w \in T'$ we have $H^1(K_{\infty,w}, E)[p^{\infty}]=H^1(K_v, E)[p^{\infty}]$ where $v$ is the prime of $K$ below $w$ and by what we just showed this is a finite group. All together this shows that $J_{T'}$ is a cofinitely generated $\Lambda$-module that is annihilated by some power of $p$.

Thus we have shown that $J=A \times B$ (decomposition as cofinitely generated $\Lambda$-modules) where $A$ is a cofinitely generated $\Zp$-module and $B$ is a torsion $\Zp$-module that is annihilated by some power of $p$. This shows that $J^*$ is a finitely generated $\Lambda$-torsion module and hence the same is true of $\coker \psi_{\infty}$.

Now we prove that $\mathfrak{S}(E/K_{\infty})$ is $\Lambda$-torsion-free. First consider the groups $Y':=\ilim_{n,m} \Selm_{p^m}(E/K_{\infty})^{\Gamma_n}$ and the group $Y:=\ilim_{n,m} \Selinf(E/K_{\infty})[p^m]^{\Gamma_n}$. We claim that $Y'$ injects into $Y$. To see this, note that the natural map from $\Selm_{p^m}(E/K_{\infty})$ to $\Selinf(E/K_{\infty})[p^m]$ has kernel $E(K_{\infty})[p^{\infty}]/p^m$. This map induces a map from $Y'$ to $Y$ with kernel $Z:=\ilim_{n,m} (E(K_{\infty})[p^{\infty}]/p^m)^{\Gamma_n}$. Since the groups $E(K_{\infty})[p^{\infty}]/p^m$ are finite and bounded as $m$ varies therefore $Z=0$ (for large enough $n$ $\Gamma_n$ acts trivially on $E(K_{\infty})[p^{\infty}]/p^m$ and hence the norm maps in the inverse limit eventually become multiplication-by-$p$). So we see that in fact $Y'$ injects into $Y$. By \cite{NSW} prop 5.5.10 $Y$ is a free $\Lambda$-module. This implies that $Y'$ is $\Lambda$-torsion-free. Now consider the commutative diagram with vertical maps induced by restriction

\begin{equation*}
\xymatrix {
Y' \ar@{^{ (}->}[r] & \ilim_{n,m} H^1(G_S(K_{\infty}), E[p^m])^{\Gamma_n}\\
\mathfrak{S}(E/K_{\infty}) \ar[u]_{\Xi} \ar@{^{ (}->}[r] & \ilim_{n,m} H^1(G_S(K_n), E[p^m]) \ar[u]_{\Xi'} }
\end{equation*}
Since $Y'$ is $\Lambda$-torsion-free, therefore to show that $\mathfrak{S}(E/K_{\infty})$ is $\Lambda$-torsion-free, it will suffice to show that the map $\Xi$ is an injection. From the commutative diagram this will be shown once we show that $\Xi'$ is an injection. We have

$$\ker \Xi'= \ilim_n \ilim_m H^1(\Gamma_n, E(K_{\infty})[p^m])$$
Since by assumption $E(K_{\infty})[p^{\infty}]$ is finite, it follows that for any $n \geq 0$ that $\ilim_m H^1(\Gamma_n, E(K_{\infty})[p^m])=0$ and hence $\ker \Xi'=0$. This completes the proof.
\end{proof}

The next 2 lemmas are the most important ingredients in our proof

\begin{lemma}\label{mu_lemma}
We have a pseudo-isomorphism $$\ilim_{n,m} \Sha^2(G_S(K_{\infty}), E[p^m])^{\Gamma_n} \sim \dot{T}_{\mu}(\Selinf(E/K_{\infty})^*)$$
\end{lemma}
\begin{proof}
The exact Kummer sequences
\begin{equation}\label{Kummer_sequence1}
0 \rightarrow E[p^m] \rightarrow E \xrightarrow{p^m} E \rightarrow0
\end{equation}
and
\begin{equation}\label{Kummer_sequence2}
0 \rightarrow E[p^m] \rightarrow E[p^{\infty}] \xrightarrow{p^m} E[p^{\infty}] \rightarrow0
\end{equation}
yield a commutative diagram

\begin{equation}
\xymatrix{
0 \ar[r] & p^m \underset{w \in S_{\infty}}{\bigoplus} H^1(K_{\infty,w}, E)[p^{\infty}] \ar[r] & \underset{w \in S_{\infty}}{\bigoplus} H^1(K_{\infty,w}, E)[p^{\infty}] \ar[r] & \underset{w \in S_{\infty}}{\bigoplus} H^2(K_{\infty,w}, E[p^m])\\
0 \ar[r] & p^m H^1(G_S(K_{\infty}), E[p^{\infty}]) \ar[u]_{\psi} \ar[r] & H^1(G_S(K_{\infty}), E[p^{\infty}]) \ar[u]_{\psi'} \ar[r]
& H^2(G_S(K_{\infty}), E[p^m]) \ar[u]_{\psi''} \ar[r] & 0
}
\end{equation}
where the 0 at the right of the lower sequence is because $H^2(G_S(K_{\infty}), E[p^{\infty}])=0$. Since $\ker \psi'=\Selinf(E/K_{\infty})$, $\ker \psi''=\Sha^2(G_S(K_{\infty}), E[p^m])$ and $\psi$ is surjective by proposition \ref{surjectivity_prop}, therefore by the snake lemma we get the following exact sequence
\begin{equation}\label{mu_sequence1}
0 \to \ker \psi \to \Selinf(E/K_{\infty}) \to \Sha^2(G_S(K_{\infty}), E[p^m]) \to 0
\end{equation}

Now consider the following commutative diagram
\begin{equation}
\xymatrix{
0 \ar[r] & \underset{w \in S_{\infty}}{\bigoplus} H^1(K_{\infty,w}, E)[p^m]\ar[r] & \underset{w \in S_{\infty}}{\bigoplus} H^1(K_{\infty,w}, E)[p^{\infty}] \ar[r]^{p^m} & p^m\underset{w \in S_{\infty}}{\bigoplus} H^1(K_{\infty,w}, E)[p^{\infty}]\\
&H^1(G_S(K_{\infty}), E[p^m]) \ar[u]_{\phi_m} \ar[r] & H^1(G_S(K_{\infty}), E[p^{\infty}]) \ar[u]_{\psi'} \ar[r]^{p^m}
& p^m H^1(G_S(K_{\infty}), E[p^{\infty}]) \ar[u]_{\psi} \ar[r] & 0
}
\end{equation}
Since $\ker \psi'=\Selinf(E/K_{\infty})$ and $\psi'$ is surjective by proposition \ref{surjectivity_prop}, therefore by the snake lemma we get an exact sequence
\begin{equation}
0 \to p^m\Selinf(E/K_{\infty}) \to \ker \psi \to \coker \phi_m \to 0
\end{equation}
This sequence in turn gives the following exact sequence
\begin{equation}\label{mu_sequence2}
0 \to \coker \phi_m \to \Selinf(E/K_{\infty})/p^m \to \Selinf(E/K_{\infty})/\ker \psi \to 0
\end{equation}
From the sequences (\ref{mu_sequence1}) and (\ref{mu_sequence2}) we get the following exact sequence
\begin{equation}
0 \to \coker \phi_m \to \Selinf(E/K_{\infty})/p^m \to \Sha^2(G_S(K_{\infty}), E[p^m]) \to 0
\end{equation}
For any $n \geq 0$ this sequence induces another exact sequence
\begin{align}\label{mu_sequence3}
\begin{aligned}
0 &\to (\coker \phi_m)^{\Gamma_n} \to (\Selinf(E/K_{\infty})/p^m)^{\Gamma_n} \to \Sha^2(G_S(K_{\infty}), E[p^m])^{\Gamma_n} \\
&\to (\coker \phi_m)_{\Gamma_n}
\end{aligned}
\end{align}
We claim that each of the terms in this exact sequence is finite. Clearly it will suffice to prove that the second term and fourth term are finite. Since $\Selinf(E/K_{\infty})^*$ is a finitely generated $\Lambda$-module (see \cite{Manin} theorem 4.5), the finiteness of the second term is easily seen by taking Pontryagin duals. Also in the proof of proposition \ref{surjectivity_prop} we have shown that $J:=\underset {w \in S_{\infty}} \bigoplus H^1(K_{\infty,w}, E)[p^{\infty}]$ is a cofinitely generated $\Lambda$-module and hence by considering Pontryagin duals $J[p^m]_{\Gamma_n}$ is finite. Since $J[p^m]_{\Gamma_n}$ surjects onto $(\coker \phi_m)_{\Gamma_n}$, it follows that the fourth term is finite. Therefore we have seen that all the terms in the sequence (\ref{mu_sequence3}) are finite so by taking inverse limits the sequence remains exact
\begin{align}\label{mu_sequence4}
\begin{aligned}
0 &\to \ilim_{n,m}(\coker \phi_m)^{\Gamma_n} \to \ilim_{n,m}(\Selinf(E/K_{\infty})/p^m)^{\Gamma_n} \to \ilim_{n,m}\Sha^2(G_S(K_{\infty}), E[p^m])^{\Gamma_n} \\
&\xrightarrow{\uptheta} \ilim_{n,m}(\coker \phi_m)_{\Gamma_n}
\end{aligned}
\end{align}
By lemma \ref{inverse_limit_lemma}(i) the second term in the above sequence is pseudo-isomorphic to $\dot{T}_{\mu}(\Selinf(E/K_{\infty})^*)$. Therefore to prove the lemma it will suffice to show that the first term and $\img \uptheta$ in the above sequence are  both finite.

First we deal with $\img \uptheta$. Consider the group $J:=\underset {w \in S_{\infty}} \bigoplus H^1(K_{\infty,w}, E)[p^{\infty}]$. By lemma \ref{inverse_limit_lemma}(ii) $\ilim_{n,m} (J[p^m])_{\Gamma_n} \sim \dot{T}_{\lambda}(J^*)$. Hence $\ilim_{n,m} (J[p^m])_{\Gamma_n}$ is a finitely generated $\Zp$-module. Now $\ilim_{n,m} (J[p^m])_{\Gamma_n}$ surjects onto $\ilim_{n,m} (\coker \phi_m)_{\Gamma_n}$ and so it follows that $\ilim_{n,m} (\coker \phi_m)_{\Gamma_n}$ is also a finitely generated $\Zp$-module.

We now prove that the group $\ilim_{n,m}\Sha^2(G_S(K_{\infty}), E[p^m])^{\Gamma_n}$ is a torsion $\Zp$-module. Taking into account the fact that $H^2(G_S(K_{\infty}), E[p^{\infty}])=0$, the Kummer sequence (\ref{Kummer_sequence2}) gives an isomorphism $H^1(G_S(K_{\infty}), E[p^{\infty}])/p^m \cong H^2(G_S(K_{\infty}), E[p^m])$. Combining the injection $\Sha^2(G_S(K_{\infty}), E[p^m])^{\Gamma_n} \hookrightarrow H^2(G_S(K_{\infty}), E[p^m])$ with this isomorphism we get an injection
\begin{equation}
\ilim_{n,m} \Sha^2(G_S(K_{\infty}), E[p^m])^{\Gamma_n} \hookrightarrow \ilim_{n,m} (H^1(G_S(K_{\infty}), E[p^{\infty}])/p^m)^{\Gamma_n}
\end{equation}
Lemma \ref{inverse_limit_lemma} shows that $\ilim_{n,m} (H^1(G_S(K_{\infty}), E[p^{\infty}])/p^m)^{\Gamma_n}$ is a $\Zp$-torsion module and hence from the injection above the same is true for $\ilim_{n,m} \Sha^2(G_S(K_{\infty}), E[p^m])^{\Gamma_n}$. Therefore $\img \uptheta$ is a $\Zp$-torsion module. It is also a finitely generated $\Zp$-module since $\ilim_{n,m} (\coker \phi_m)_{\Gamma_n}$ is finitely generated over $\Zp$ as shown above. This implies that $\img \uptheta$ is finite as claimed.

We now deal with $\ilim_{n,m} (\coker \phi_m)^{\Gamma_n}$. We will actually show this group is trivial. Let $J:=\underset {w \in S_{\infty}} \bigoplus H^1(K_{\infty,w}, E)[p^{\infty}]$. In the proof of proposition \ref{surjectivity_prop} we have shown that $J=A \times B$ where $A$ is a cofinitely generated $\Zp$-module and $B$ is a torsion $\Zp$-module that is annihilated by $p^t$ for some $t$. For any $m \geq 0, \coker \phi_m$ is the quotient of $J[p^m]=A[p^m] \times B[p^m]$. Let $\alpha=(\alpha_{n,m}) \in \ilim_{n,m} (\coker \phi_m)^{\Gamma_n}$. Note that the transition map on the first index is the norm map and on the second index it is multiplication-by-$p$. For each $(n,m) \in \Z_{\geq 0} \times \Z_{\geq 1}$ choose $a_{n,m} \in A[p^m]$ and $b_{n,m} \in B[p^m]$ such that $\alpha_{n,m}$ is represented by $(a_{n,m}, b_{n,m})$

Now let $(n,m) \in \Z_{\geq 0} \times \Z_{\geq 1}$. We will show that $\alpha_{n,m}=0$. Recall that $p^t$ annihilates $B$. Consider $(a_{n,m'}, b_{n,m'})$ where $m'=m+t$. We claim that $(a_{n,m'}, b_{n,m'}) \equiv (0, b')$ for some $b' \in B[p^{m'}]$ (the congruence is modulo $\img \phi_{m'}$). To see this, note that since $A$ is cofinitely generated over $\Zp$, therefore $A[p^{m'}]$ is finite and hence is fixed by $\Gamma_{n'}$ for some $n' \geq n$. Since for any $n'' > n'$ we have $\text{Tr}_{K_{n''}/K_{n'}}(\alpha_{n'',m'})=\alpha_{n',m'}$ and $\Gamma_{n'}$ acts trivially on $A[p^{m'}]$, therefore by considering large enough $n'' > n'$ we easily see that $(a_{n',m'}, b_{n',m'}) \equiv (0, b'')$ for some $b'' \in B[p^{m'}]$ and hence $(a_{n,m'}, b_{n,m'}) \equiv (0, b')$ for some $b' \in B[p^{m'}]$ as claimed.

Now we have that $p^t \alpha_{n,m'}=\alpha_{n,m}$. Since $(a_{n,m'}, b_{n,m'}) \equiv (0, b')$ and $p^t$ annihilates $B$, therefore we see that $(a_{n,m}, b_{n,m}) \equiv (0,0)$ i.e. $\alpha_{n,m}=0$. This proves that $\alpha=0$ thus showing that $\ilim_{n,m} (\coker \phi_m)^{\Gamma_n}$ is trivial. This completes the proof of the lemma.
\end{proof}

We now define $K_{n,m}$ to be the kernel of the map (induced by restriction) $H^1(G_S(K_{\infty}), E[p^m])_{\Gamma_n} \to \big(\underset{w \in S_{\infty}}{\bigoplus} H^1(K_{\infty,w}, E[p^m])\big)_{\Gamma_n}$
\begin{lemma}\label{lambda_lemma}
We have a pseudo-isomorphism $$\ilim_{n,m} K_{n,m} \sim \dot{T}_{\lambda}(\Selinf(E/K_{\infty})^*)$$
\end{lemma}
\begin{proof}
The exact Kummer sequences
\begin{equation}\label{Kummer_sequence3}
0 \rightarrow E[p^m] \rightarrow E \xrightarrow{p^m} E \rightarrow0
\end{equation}
and
\begin{equation}\label{Kummer_sequence4}
0 \rightarrow E[p^m] \rightarrow E[p^{\infty}] \xrightarrow{p^m} E[p^{\infty}] \rightarrow0
\end{equation}
yield a commutative diagram

\begin{equation}
\xymatrix{
0 \ar[r] & \underset{w \in S_{\infty}}{\bigoplus} E(K_{\infty,w})/p^m \ar[r] & \underset{w \in S_{\infty}}{\bigoplus} H^1(K_{\infty,w}, E[p^m]) \ar[r] & \underset{w \in S_{\infty}}{\bigoplus} H^1(K_{\infty,w}, E)[p^m] \ar[r] & 0\\
0 \ar[r] & E(K_{\infty})[p^{\infty}]/p^m \ar[u]_{\phi_m} \ar[r] & H^1(G_S(K_{\infty}), E[p^m]) \ar[u]_{\psi_m} \ar[r]
& H^1(G_S(K_{\infty}), E[p^{\infty}])[p^m] \ar[u]_{\psi'_m} \ar[r] & 0
}
\end{equation}

Taking $\Gamma_n$-coinvariants we get

\begin{equation}\label{lambda_diagram}
\xymatrix{
0 \ar[r] & B_{n,m} \ar[r] & \big(\underset{w \in S_{\infty}}{\bigoplus} H^1(K_{\infty,w}, E[p^m])\big)_{\Gamma_n} \ar[r] & \big(\underset{w \in S_{\infty}}{\bigoplus} H^1(K_{\infty,w}, E)[p^m]\big)_{\Gamma_n} \ar[r] & 0\\
0 \ar[r] & A_{n,m} \ar[u]_{\phi_{n,m}} \ar[r] & H^1(G_S(K_{\infty}), E[p^m])_{\Gamma_n} \ar[u]_{\psi_{n,m}} \ar[r]
& H^1(G_S(K_{\infty}), E[p^{\infty}])[p^m]_{\Gamma_n} \ar[u]_{\psi'_{n,m}} \ar[r] & 0
}
\end{equation}
where $A_{n,m}$ is the image of the map $$(E(K_{\infty})[p^{\infty}]/p^m)_{\Gamma_n} \to H^1(G_S(K_{\infty}), E[p^m])_{\Gamma_n}$$ and $B_{n,m}$ is the image of the map $$\big(\underset{w \in S_{\infty}}{\bigoplus} E(K_{\infty,w})/p^m\big)_{\Gamma_n} \to \big(\underset{w \in S_{\infty}}{\bigoplus} H^1(K_{\infty,w}, E[p^m])\big)_{\Gamma_n}$$

Applying the snake lemma to the diagram (\ref{lambda_diagram}) we get an exact sequence
\begin{equation}\label{lambda_sequence1}
0 \to \ker \phi_{n,m} \to K_{n,m} \to \ker \psi'_{n,m} \to \coker \phi_{n,m}
\end{equation}
We claim that each of the terms in this exact sequence is finite. To simplify arguments, we prove this fact subject to the condition that $n \geq N$ where $N \geq 0$ is an integer such that $K_{\infty}/K_N$ is totally ramified at all primes of $K_N$ above $p$ and such that every prime of $K_N$ above a prime in $S$ that does not split completely in $K_{\infty}/K_N$ is inert in this extension. Clearly it will suffice to prove that the first, third and fourth terms are finite. In what follows all inverse limits with indices involving $n$ will be taken over $n \geq N$

Since $E(K_{\infty})[p^{\infty}]/p^m$ is finite, therefore $\ker \phi_{n,m}$ is finite. Moreover, the order of $E(K_{\infty})[p^{\infty}]/p^m$ is bounded as $m$ varies. Hence the order of $(E(K_{\infty})[p^{\infty}]/p^m)_{\Gamma_n}$ is bounded as $n$ and $m$ vary. This in turn shows that the order of $\ker \phi_{n,m}$ is bounded as $n$ and $m$ vary. It follows that $\ilim_{n,m} \ker \phi_{n,m}$ is finite. This last fact will be needed later.

Now we show that $\coker \phi_{n,m}$ is finite by showing that $D(n,m):=\big(\underset{w \in S_{\infty}}{\bigoplus} E(K_{\infty,w})/p^m\big)_{\Gamma_n}$ is finite. We will also show that $\ilim_{n,m} \coker \phi_{n,m}$ is finite as we will need this later. We write $S_{\infty} = S_p \, \dotcup \, S_{\text{split}} \, \dotcup \, S_{\text{nsplit}}$ where $S_p$ is the set of primes of $S_{\infty}$ above $p$ and $S_{\text{split}}$ is the set of primes in $S_{\infty}$ above those in $S$ that split completely in $K_{\infty}/K$. First we show that $D_p(n,m):=\big(\underset{w \in S_p}{\bigoplus} E(K_{\infty,w})/p^m\big)_{\Gamma_n}$ is finite. Let $w \in S_p$. Note that by the condition on $n$, $\Gamma_n$ acts on $E(K_{\infty,w})/p^m$. We let $\text{Tor}(E(K_{\infty,w})$ be the $\Z$-torsion subgroup of $E(K_{\infty,w})$ and $E(K_{\infty,w})_{\text{Tor}}:=E(K_{\infty,w})/\text{Tor}(E(K_{\infty,w}))$. We have an exact sequence
\begin{equation}\label{lambda_sequence2}
((\text{Tor}(E(K_{\infty,w})))/p^m)_{\Gamma_n} \to (E(K_{\infty,w})/p^m)_{\Gamma_n} \to (E(K_{\infty,w})_{\text{Tor}}/p^m)_{\Gamma_n} \to 0
\end{equation}
Note that the Pontryagin dual of $E(K_{\infty,w})[p^{\infty}]$ is a finitely generated $\Zp[[\Gamma_N]]$-module. Since $\text{Tor}(E(K_{\infty,w}))/p^m=E(K_{\infty,w})[p^{\infty}]/p^m$, therefore by considering Pontryagin duals it follows that $((\text{Tor}(E(K_{\infty,w})))/p^m)_{\Gamma_n}$ is finite. Also from lemma \ref{inverse_limit_lemma} it follows that $\ilim_{n,m} ((\text{Tor}(E(K_{\infty,w})))/p^m)_{\Gamma_n}$ is finite.

Now we turn to the group $(E(K_{\infty,w})_{\text{Tor}}/p^m)_{\Gamma_n}$. Since $E(K_{\infty,w})_{\text{Tor}}$ is torsion-free, therefore it follows that $(E(K_{\infty,w})_{\text{Tor}}\otimes \Qp/\Zp)[p^m]=E(K_{\infty,w})_{\text{Tor}}/p^m$. Also note that $E(K_{\infty,w})_{\text{Tor}}\otimes \Qp/\Zp = E(K_{\infty,w})\otimes \Qp/\Zp$. These 2 facts show that
\begin{equation}\label{lambda_equality}
((E(K_{\infty,w})\otimes \Qp/\Zp)[p^m])_{\Gamma_n}=(E(K_{\infty,w})_{\text{Tor}}/p^m)_{\Gamma_n}
\end{equation}
Wingberg (\cite{Wingberg2} theorem 2.2) has shown that $(E(K_{\infty,w})\otimes \Qp/\Zp)^*$ is pseudo-isomorphic to a finitely generated free $\Zp[[\Gamma_N]]$-module. This together with the equality (\ref{lambda_equality}) show that $(E(K_{\infty,w})_{\text{Tor}}/p^m)_{\Gamma_n}$ is finite. Using lemma \ref{inverse_limit_lemma} we also see that $\ilim_{n,m}(E(K_{\infty,w})_{\text{Tor}}/p^m)_{\Gamma_n}$ is finite. From the facts above and the exact sequence (\ref{lambda_sequence2}) it follows that $(E(K_{\infty,w})/p^n)_{\Gamma_n}$ is finite and that $\ilim_{n,m} (E(K_{\infty,w})/p^m)_{\Gamma_n}$ is finite. Thus we have shown that for any $n \geq N, m \geq 1$ that $D_p(n,m)$ is finite and $\ilim_{n,m} D_p(n,m)$ is also finite.

Now we turn to the group $D_{\text{nsplit}}(n,m):=\big(\underset{w \in S_{\text{nsplit}}}{\bigoplus} E(K_{\infty,w})/p^m\big)_{\Gamma_n}$. Let $w \in S_{\text{nsplit}}$. We claim that $E(K_{\infty,w})\otimes \Qp/\Zp=0$. To see this, it will suffice to show for any $t \geq 0$ that $E(K_{t,w}) \otimes \Qp/\Zp=0$. By Mattuck's theorem $E(K_{t,w})=\Z_l^{[K_{t,w}:\Q_l]}\times T$ where $l \neq p$ is the characteristic of the residue field of $K_{t,w}$ and $T$ is a finite group. It follows from this that $E(K_{t,w}) \otimes \Qp/\Zp =0$ and hence $E(K_{\infty, w}) \otimes \Qp/\Zp = 0$ as claimed. Then just as in the case of $D_p(n,m)$, we also have $D_{\text{nsplit}}(n,m)$ is finite and $\ilim_{n,m} D_{\text{nsplit}}(n,m)$ is also finite.

Finally we turn to the group $D_{\text{split}}(n,m):=\big(\underset{w \in S_{\text{split}}}{\bigoplus} E(K_{\infty,w})/p^m\big)_{\Gamma_n}$. Let $v \in S$ be a prime that splits completely in $K_{\infty}/K$ and define $C_v:=\underset{w | v}{\bigoplus} E(K_{\infty,w})/p^m$ where the sum runs over all primes of $K_{\infty}$ lying over $v$. We claim that $H^1(\Gamma_n, C_v)=(C_v)_{\Gamma_n}=0$. To simplify matters, we prove this for $n=0$ (for arbitrary $n$ the proof is similar). The group $C_v$ is a direct limit of induced $\Gamma$-modules and hence by Shapiro's lemma it follows that $H^1(\Gamma, C_v)=H^1(\{1\}, E(K_v)/p^m)=0$. Thus we see that $D_{\text{split}}(n,m)=0$.

All in all, we see from the above for any $n \geq N, m \geq 1$ that $D(n,m)$ is finite and that $\ilim_{n,m} D(n,m)$ is also finite. It follows that $\coker \phi_{n,m}$ is finite and that $\ilim_{n,m} \coker \phi_{n,m}$ is also finite.

Finally we prove that $\ker \psi'_{n,m}$ is finite. Since $\ker \psi'_{n,m} \subseteq H^1(G_S(K_{\infty}), E[p^{\infty}])[p^m]_{\Gamma_n}$ we easily see that by taking Pontryagin duals that it suffices to prove that $H^1(G_S(K_{\infty}), E[p^{\infty}])^*$ is a finitely generated $\Lambda$-module. To show this last fact we have to show that $H^1(G_S(K_{\infty}), E[p^{\infty}])^{\Gamma}$ is cofinitely generated over $\Zp$. Since $cd_p(\Gamma)=1$, therefore we have a surjection $H^1(G_S(K), E[p^{\infty}]) \twoheadrightarrow H^1(G_S(K_{\infty}), E[p^{\infty}])^{\Gamma}$ so it suffices to prove that $H^1(G_S(K), E[p^{\infty}])$ is cofinitely generated over $\Zp$ i.e. we must show that $H^1(G_S(K), E[p^{\infty}])[p]$ is finite. But $H^1(G_S(K), E[p])$ is finite (see \cite{NSW} theorem 8.3.20) and this group surjects onto $H^1(G_S(K), E[p^{\infty}])[p]$ so it indeed follows that $H^1(G_S(K), E[p^{\infty}])$ is cofinitely generated over $\Zp$. This proves that $\ker \psi'_{n,m}$ is finite.

We have now shown that each of the terms in the exact sequence (\ref{lambda_sequence1}) are finite and so by taking inverse limits the sequence remains exact

\begin{equation}\label{lambda_sequence3}
0 \to \ilim_{n,m} \ker \phi_{n,m} \to \ilim_{n,m} K_{n,m} \to \ilim_{n,m} \ker \psi'_{n,m} \to \ilim_{n,m} \coker \phi_{n,m}
\end{equation}

We have shown above that $\ilim_{n,m} \ker \phi_{n,m}$ and $\ilim_{n,m} \coker \phi_{n,m}$ are both finite so from the exact sequence (\ref{lambda_sequence3}) we get a pseudo-isomorphism $\ilim_{n,m} K_{n,m} \sim \ilim_{n,m} \ker \psi'_{n,m}$. Therefore to prove the lemma we only have to show that $\ilim_{n,m} \ker \psi'_{n,m}$  is pseudo-isomorphic to $\dot{T}_{\lambda}(\Selinf(E/K_{\infty})^*)$. To show this, first consider the exact sequence

\begin{equation}
0 \rightarrow \Selinf(E/K_{\infty}) \rightarrow H^1(G_S(K_{\infty}), E[p^{\infty}]) \xrightarrow{\psi'} \underset{w \in S_{\infty}}{\bigoplus} H^1(K_{\infty,w}, E) \rightarrow 0
\end{equation}
The surjectivity of $\psi'$ is due to proposition \ref{surjectivity_prop}. This exact sequence induces another sequence

\begin{align}
\begin{aligned}
0 &\rightarrow \Selinf(E/K_{\infty})[p^m] \rightarrow H^1(G_S(K_{\infty}), E[p^{\infty}])[p^m] \xrightarrow{\psi'_m} \underset{w \in S_{\infty}}{\bigoplus} H^1(K_{\infty,w}, E)[p^m]\\
 &\xrightarrow{\psi''_m} \Selinf(E/K_{\infty})/p^m
\end{aligned}
\end{align}
We break this sequence into 2 exact sequences

\begin{equation}
0 \rightarrow \Selinf(E/K_{\infty})[p^m] \rightarrow H^1(G_S(K_{\infty}), E[p^{\infty}])[p^m] \xrightarrow{\phi_m} \img \psi'_m \rightarrow 0
\end{equation}

\begin{equation}
0 \rightarrow \img \psi'_m \xrightarrow{\theta_m} \underset{w \in S_{\infty}}{\bigoplus} H^1(K_{\infty,w}, E)[p^m] \rightarrow \img \psi''_m \rightarrow 0
\end{equation}
Taking $\Gamma_n$-coinvaraints of both these sequences we get
\begin{align}\label{lambda_sequence4}
\begin{aligned}
(\img \psi'_m)^{\Gamma_n} &\rightarrow \Selinf(E/K_{\infty})[p^m]_{\Gamma_n} \rightarrow H^1(G_S(K_{\infty}), E[p^{\infty}])[p^m]_{\Gamma_n}\\ &\xrightarrow{\phi_{n,m}} (\img \psi'_m)_{\Gamma_n} \rightarrow 0
\end{aligned}
\end{align}

\begin{equation}\label{lambda_sequence5}
(\img \psi''_m)^{\Gamma_n} \rightarrow (\img \psi'_m)_{\Gamma_n} \xrightarrow{\theta_{n,m}} \big(\underset{w \in S_{\infty}}{\bigoplus} H^1(K_{\infty,w}, E)[p^m]\big)_{\Gamma_n}
\end{equation}
We claim the each of the terms in the sequences (\ref{lambda_sequence4}) and (\ref{lambda_sequence5}) is finite. First we deal with (\ref{lambda_sequence4}): Let  $J:=\underset{w \in S_{\infty}}{\bigoplus} H^1(K_{\infty,w}, E)[p^{\infty}]$. Then $(\img \psi'_m)^{\Gamma_n} \subseteq J[p^m]^{\Gamma_n}$. In the proof of proposition \ref{surjectivity_prop} we proved that $J$ is a cofinitely generated $\Lambda$-module. It follows from this that $J[p^m]^{\Gamma_n}$ is finite and hence $(\img \psi'_m)^{\Gamma_n}$ is also finite. We showed above that $H^1(G_S(K_{\infty}), E[p^{\infty}])$ is a cofinitely generated $\Lambda$-module. It follows that $H^1(G_S(K_{\infty}), E[p^{\infty}])[p^m]_{\Gamma_n}$ is finite. This proves that all the terms in the sequence (\ref{lambda_sequence4}) are finite.

Now we deal with the sequence (\ref{lambda_sequence5}). We know that $\Selinf(E/K_{\infty})^*$ is a finitely generated $\Lambda$-module (see \cite{Manin} theorem 4.5) so it follows that $(\Selinf(E/K_{\infty})/p^m)^{\Gamma_n}$ is finite. Since $(\img \psi''_m)^{\Gamma_n} \subseteq (\Selinf(E/K_{\infty})/p^m)^{\Gamma_n}$, it follows that $(\img \psi''_m)^{\Gamma_n}$ is also finite. Also since $J$ is a cofinitely generated $\Lambda$-module, it follows that $J[p^m]_{\Gamma_n}=\big(\underset{w \in S_{\infty}}{\bigoplus} H^1(K_{\infty,w}, E)[p^m]\big)_{\Gamma_n}$ is finite. Therefore all the terms in the sequence (\ref{lambda_sequence5}) are finite.

As all the terms in the sequence (\ref{lambda_sequence4}) and (\ref{lambda_sequence5}) are finite, therefore by taking inverse limits the sequences remain exact

\begin{align}\label{lambda_sequence6}
\begin{aligned}
\ilim_{n,m} (\img \psi'_m)^{\Gamma_n} &\rightarrow \ilim_{n,m} \Selinf(E/K_{\infty})[p^m]_{\Gamma_n} \rightarrow \ilim_{n,m} H^1(G_S(K_{\infty}), E[p^{\infty}])[p^m]_{\Gamma_n}\\ &\xrightarrow{\Phi} \ilim_{n,m}(\img \psi'_m)_{\Gamma_n} \rightarrow 0
\end{aligned}
\end{align}

\begin{equation}\label{lambda_sequence7}
\ilim_{n,m} (\img \psi''_m)^{\Gamma_n} \rightarrow \ilim_{n,m} (\img \psi'_m)_{\Gamma_n} \xrightarrow{\Theta} \ilim_{n,m} \big(\underset{w \in S_{\infty}}{\bigoplus} H^1(K_{\infty,w}, E)[p^m]\big)_{\Gamma_n}
\end{equation}
We have an exact sequence
\begin{equation}\label{lambda_sequence8}
0 \rightarrow \ker(\Phi) \rightarrow \ker(\Theta \circ \Phi) \xrightarrow{\Phi} \ker(\Theta)
\end{equation}
Note that $\ker(\Theta \circ \Phi)=\ilim_{n,m} \ker \psi'_{n,m}$. Therefore from the exact sequence, we see that to prove the lemma, we only have to show that $\ker \Phi$ is pseudo-isomorphic to $\dot{T}_{\lambda}(\Selinf(E/K_{\infty})^*)$ and that $\Phi(\ker(\Theta \circ \Phi))$ is finite.

First we deal with $\ker \Phi$. From lemma \ref{inverse_limit_lemma} $\ilim_{n,m} \Selinf(E/K_{\infty})[p^m]_{\Gamma_n}$ is pseudo-isomorphic to $\dot{T}_{\lambda}(\Selinf(E/K_{\infty})^*)$. Therefore from the exact sequence (\ref{lambda_sequence6}), to show the existence of a pseudo-isomorphism $\ker \Phi \sim \dot{T}_{\lambda}(\Selinf(E/K_{\infty})^*)$, we see that it will suffice to show that $\ilim_{n,m} (\img \psi'_m)^{\Gamma_n}=0$. Let  $J:=\underset{w \in S_{\infty}}{\bigoplus} H^1(K_{\infty,w}, E)[p^{\infty}]$. Then $\ilim_{n,m} (\img \psi'_m)^{\Gamma_n} \subseteq \ilim_{n,m} J[p^m]^{\Gamma_n}$. According to \cite{NSW} prop. 5.5.10 $\ilim_{n,m} J[p^m]^{\Gamma_n}$ is a free $\Lambda$-module with the same rank as $J^*$. In the proof of proposition \ref{surjectivity_prop} we showed that $J^*$ is a torsion $\Lambda$-module. Therefore it follows that $\ilim_{n,m} J[p^m]^{\Gamma_n}=0$ and hence also $\ilim_{n,m} (\img \psi'_m)^{\Gamma_n}=0$. This proves that $\ker \Phi$ is pseudo-isomorphic to $\dot{T}_{\lambda}(\Selinf(E/K_{\infty})^*)$.

Now we show that $\Phi(\ker(\Theta \circ \Phi))$ is finite. First we show that this group is finitely generated over $\Zp$. Recall that we showed above that $\ker(\Theta \circ \Phi)=\ilim_{n,m} \ker \psi'_{n,m}$ is pseudo-isomorphic to $\ilim_{n,m} K_{n,m}$. We have $\ilim_{n,m} K_{n,m} \subseteq \ilim_{n,m} H^1(G_S(K_{\infty}), E[p^m])_{\Gamma_n}$. Therefore to show that $\Phi(\ker(\Theta \circ \Phi))$ is finitely generated over $\Zp$, it will suffice to show that $\ilim_{n,m} H^1(G_S(K_{\infty}), E[p^m])_{\Gamma_n}$ is finitely generated over $\Zp$. Consider the bottom row of the commutative diagram (\ref{lambda_diagram})

$$(E(K_{\infty})[p^{\infty}]/p^m)_{\Gamma_n} \to H^1(G_S(K_{\infty}), E[p^m])_{\Gamma_n} \to H^1(G_S(K_{\infty}), E[p^{\infty}])[p^m]_{\Gamma_n} \to 0$$\\
We showed above that all the terms in the above exact sequence are finite. Therefore by taking inverse limits the sequence remains exact

$$\ilim_{n,m} (E(K_{\infty})[p^{\infty}]/p^m)_{\Gamma_n} \to \ilim_{n,m} H^1(G_S(K_{\infty}), E[p^m])_{\Gamma_n} \to \ilim_{n,m} H^1(G_S(K_{\infty}), E[p^{\infty}])[p^m]_{\Gamma_n} \to 0$$

Applying lemma \ref{inverse_limit_lemma} to the first and last terms of the above sequence, it follows that $\ilim_{n,m} H^1(G_S(K_{\infty}), E[p^m])_{\Gamma_n}$ is in fact finitely generated over $\Zp$ as desired.

Now we show that $\Phi(\ker(\Theta \circ \Phi))$ is a $\Zp$-torsion module. From the exact sequence (\ref{lambda_sequence8}), this will follow if we can prove that $\ker \Theta$ is a torsion $\Zp$-module. From the exact sequence (\ref{lambda_sequence7}), this will follow once we show that $\ilim_{n,m} (\img \psi''_m)^{\Gamma_n}$ has this property. But $\ilim_{n,m} (\img \psi''_m)^{\Gamma_n} \subseteq (\Selinf(E/K_{\infty})/p^m)^{\Gamma_n}$ and so the desired result follows from lemma \ref{inverse_limit_lemma}.

Thus we have shown that $\Phi(\ker(\Theta \circ \Phi))$ is a finitely generated torsion $\Zp$-module. It follows that this group is finite. This completes the proof of the lemma.
\end{proof}

Assume that we have a first quadrant spectral sequence $E_r^{st} \Rightarrow E^{s+t}$ such that $E_2^{s,t}=0$ for all $s > 1$ and all $t$, then we have an exact sequence (see \cite{NSW} lemma 2.1.3 and \cite{CE} prop. XV-5.7)
$$0 \to E_2^{1,1} \to E^2 \to E_2^{0,2} \to 0$$\\
Now let $n,m \geq 0$. Since $cd_p(\Gamma_n)=1$, therefore we can apply the above result to the Hochschild-Serre spectral sequence $H^s(\Gamma_n, H^t(G_S(K_{\infty}), E[p^m])) \Rightarrow H^{s+t}(G_S(K_n), E[p^m])$ (where $K_n$ is the fixed field of $\Gamma_n$) to get an exact sequence
\begin{equation}\label{HS_sequence}
0 \rightarrow H^1(G_S(K_{\infty}), E[p^m])_{\Gamma_n} \xrightarrow{f_{n,m}} H^2(G_S(K_n), E[p^m]) \xrightarrow{g_{n,m}} H^2(G_S(K_{\infty}), E[p^m])^{\Gamma_n} \rightarrow 0
\end{equation}

\begin{lemma}\label{SS_commutative_diagram_lemma}
For $n' > n$ we have a commutative diagram

\begin{equation*}
\xymatrix{
H^1(G_S(K_{\infty}), E[p^m])_{\Gamma_{n'}} \ar[d]_{\can} \ar[r]^{f_{n',m}} & H^2(G_S(K_{n'}), E[p^m]) \ar[d]_{\cor} \ar[r]^{g_{n',m}} & H^2(G_S(K_{\infty}), E[p^m])^{\Gamma_{n'}} \ar[d]_{\norm}\\
H^1(G_S(K_{\infty}), E[p^m])_{\Gamma_n} \ar[r]^{f_{n,m}} & H^2(G_S(K_n), E[p^m]) \ar[r]^{g_{n,m}} & H^2(G_S(K_{\infty}), E[p^m])^{\Gamma_n}
}
\end{equation*}
where the maps $can$, $cor$ and $norm$ are the canonical projection, corestriction and norm, respectively.
Also for $m' > m$ we have a commutative diagram

\begin{equation*}
\xymatrix{
H^1(G_S(K_{\infty}), E[p^{m'}])_{\Gamma_n} \ar[d]_{p^{m'-m}} \ar[r]^{f_{n,m'}} & H^2(G_S(K_n), E[p^{m'}]) \ar[d]_{p^{m'-m}} \ar[r]^{g_{n,m'}} & H^2(G_S(K_{\infty}), E[p^{m'}])^{\Gamma_n} \ar[d]_{p^{m'-m}}\\
H^1(G_S(K_{\infty}), E[p^m])_{\Gamma_n} \ar[r]^{f_{n,m}} & H^2(G_S(K_n), E[p^m]) \ar[r]^{g_{n,m}} & H^2(G_S(K_{\infty}), E[p^m])^{\Gamma_n}
}
\end{equation*}
where the vertical maps are induced by multiplication by $p^{m'-m}$

\end{lemma}
\begin{proof}
From the formula for the corestriction map (\cite{Weiss} prop. 2.5.2) it is easy to show that the corestriction map $\cor: H^1(\Gamma_{n'}, H^1(G_S(K_{\infty}), E[p^m])) \to H^1(\Gamma_n, H^1(G_S(K_{\infty}), E[p^m]))$ corresponds to the canonical projection $\can: H^1(G_S(K_{\infty}), E[p^m])_{\Gamma_{n'}}\to H^1(G_S(K_{\infty}), E[p^m])_{\Gamma_n}$. Also we know that the corestriction map $\cor: H^2(G_S(K_{\infty}), E[p^m])^{\Gamma_{n'}} \to H^2(G_S(K_{\infty}), E[p^m])^{\Gamma_n}$ is equal to the norm map. Therefore we see that to show that the first diagram commutes, it suffices to show the commutativity of the following diagram

\begin{equation}\label{SS_diagram}
\xymatrix{
H^1(\Gamma_{n'}, H^1(G_S(K_{\infty}), E[p^m])) \ar[d]_{\cor} \ar[r]^-{f_{n',m}} & H^2(G_S(K_{n'}), E[p^m]) \ar[d]_{\cor} \ar[r]^{g_{n',m}} & H^2(G_S(K_{\infty}), E[p^m])^{\Gamma_{n'}} \ar[d]_{\cor}\\
H^1(\Gamma_n, H^1(G_S(K_{\infty}), E[p^m])) \ar[r]^-{f_{n,m}} & H^2(G_S(K_n), E[p^m]) \ar[r]^{g_{n,m}} & H^2(G_S(K_{\infty}), E[p^m])^{\Gamma_n}
}
\end{equation}
As in the proof of \cite{NSW} theorem 2.4.1, the data $(G_S(K_{n'}), \Gamma_{n'}, E[p^m])$ determines a double complex $D(G_S(K_{n'}), \Gamma_{n'}, E[p^m])$. Similarly, we have a double complex $D(G_S(K_n), \Gamma_n, E[p^m])$. By taking the column-wise filtrations of the total complexes of these double complexes we obtain 2 Hochschild-Serre spectral sequences $SS(G_S(K_{n'}), \Gamma_{n'}, E[p^m])$ and $SS(G_S(K_n), \Gamma_n, E[p^m])$ (see
\cite{McCleary} theorem 2.15 and \cite{NSW} theorem 2.4.1). The obvious corestriction map on cochains induces a morphism of double complex $\cor: D(G_S(K_{n'}), \Gamma_{n'}, E[p^m]) \to D(G_S(K_n, \Gamma_n, E[p^m])$ which, in turn, induces a morphism of spectral sequences $\cor: SS(G_S(K_{n'}), \Gamma_{n'}, E[p^m]) \to SS(G_S(K_n, \Gamma_n, E[p^m])$ and their corresponding limit terms.

One checks using the definitions of $f_{n,m}$ and $g_{n,m}$ (see \cite{CE} prop. XV-5.7) that these morphisms commute with the induced maps on the terms in the diagram (\ref{SS_diagram}) and that these induced maps are actually corestriction. This proves that the diagram (\ref{SS_diagram}) commutes and hence the first diagram in the statement of the lemma commutes.

To show that the second diagram in the statement of the lemma commutes one argues similar fashion using the map $p^{m'-m}: D(G_S(K_n), \Gamma_n, E[p^{m'}]) \to D(G_S(K_n), \Gamma_n, E[p^m])$ which is induced by multiplication by $p^{m'-m}$

Alternatively, to show that the 2 diagrams commute, one can use the explicit descriptions of $f_{n,m}$ and $g_{n,m}$: the map $g_{n,m}$ is the restriction map (see \cite{Maclane} prop. XI-10.2) whereas the map $f_{n,m}^{-1}: \img f_{n,m} \to H^1(G_S(K_{\infty}), E[p^m])_{\Gamma_n}$ is described in \cite{DHW} (see also \cite{Huebschmann} theorem 2).
\end{proof}
Now let $v$ be a prime of $K_n$, $w$ a prime of $K_{\infty}$ above $v$ and $\Gamma_{n,w}$ the decomposition group of $w$ in $K_{\infty}/K$. We have a Hochschild-Serre spectral sequence $H^s(\Gamma_{n,w}, H^t(K_{\infty,w}, E[p^m])) \Rightarrow H^{s+t}(K_{n,v}, E[p^m])$. We certainly have $cd_p(\Gamma_{n,w})=1$ so as in the the global case we get an exact sequence
\begin{equation}
0 \rightarrow H^1(K_{\infty,w}, E[p^m])_{\Gamma_{n,w}} \xrightarrow{f_{w,n,m}} H^2(K_{n,v}, E[p^m]) \xrightarrow{g_{w,n,m}} H^2(K_{\infty,w}, E[p^m])^{\Gamma_{n,w}} \rightarrow 0
\end{equation}
By Shapiro's lemma $H^1(K_{\infty,w}, E[p^m])_{\Gamma_{n,w}}=(\bigoplus_{w|v} H^1(K_{\infty,w}, E[p^m]))_{\Gamma_n}$ and $H^2(K_{\infty,w}, E[p^m])^{\Gamma_{n,w}}=(\bigoplus_{w|v} H^2(K_{\infty,w}, E[p^m]))^{\Gamma_n}$ where the direct sum runs over all primes of $w$ dividing $v$. Taking direct sums, we get a diagram

\begin{equation*}
\resizebox{\displaywidth+3cm}{!}{
\xymatrix{
0 \ar[r] & (\underset{w \in S_{\infty}}{\bigoplus} H^1(K_{\infty,w}, E[p^m]))_{\Gamma_n} \ar[r]^-{f'_{n,m}} & \underset{v \in S_n}{\bigoplus} H^2(K_{n,v}, E[p^m]) \ar[r]^-{g'_{n,m}} & (\underset{w \in S_{\infty}}{\bigoplus} H^2(K_{\infty,w}, E[p^m]))^{\Gamma_n} \ar[r] & 0\\
0 \ar[r] & H^1(G_S(K_{\infty}), E[p^m])_{\Gamma_n} \ar[u]^{\psi_{n,m}} \ar[r]^{f_{n,m}} & H^2(G_S(K_n), E[p^m]) \ar[u]^{\psi'_{n,m}} \ar[r]^{g_{n,m}} & H^2(G_S(K_{\infty}), E[p^m])^{\Gamma_n} \ar[u]^{\psi''_{n,m}} \ar[r] & 0\\
}}
\end{equation*}

The vertical maps in the diagram are induced by restriction. This diagram commutes. To see this, one argues as in the proof of lemma \ref{SS_commutative_diagram_lemma} taking the restriction map of the appropriate double complexes. Applying the snake lemma we get an exact sequence

\begin{equation}
0 \to \ker \psi_{n,m} \to \ker \psi'_{n,m} \to \ker \psi''_{n,m} \to \coker \psi_{n,m}
\end{equation}\\
Lemma \ref{SS_commutative_diagram_lemma} and its local analog allow us to take inverse limits of this exact sequence. By \cite{NSW} theorem 8.3.20 $H^2(G_S(K_n), E[p^m])$ is finite and therefore $\ker \psi'_{n,m}$ is also finite. Therefore by taking inverse limits the sequence remains exact

\begin{equation}\label{final_sequence}
0 \to \ilim_{n,m} \ker \psi_{n,m} \to \ilim_{n,m} \ker \psi'_{n,m} \to \ilim_{n,m} \ker \psi''_{n,m} \xrightarrow{\theta} \ilim_{n,m}\coker \psi_{n,m}
\end{equation}

By lemma \ref{lambda_lemma} $\ilim_{n,m} \ker \psi_{n,m} \sim \dot{T}_{\lambda}(\Selinf(E/K_{\infty})^*)$. Also note that $\ilim_{n,m} \ker \psi''_{n,m} = \ilim_{n,m} \Sha^2(G_S(K_{\infty}), E[p^m])^{\Gamma_n}$ so by lemma \ref{mu_lemma} we have $\ilim_{n,m} \ker \psi''_{n,m} \sim \dot{T}_{\mu}(\Selinf(E/K_{\infty})^*)$. We claim that $\img \theta$ is finite. Since $\ilim_{n,m} \ker \psi''_{n,m} \sim \dot{T}_{\mu}(\Selinf(E/K_{\infty})^*)$ therefore $\img \theta$ is a $\Zp$-torsion module. Define $J:=\underset {w \in S_{\infty}} \bigoplus H^1(K_{\infty,w}, E)[p^{\infty}]$. Then we have an exact sequence
$$0 \to \underset{w \in S_{\infty}}{\bigoplus} E(K_{\infty, w})/p^m \to \underset{w \in S_{\infty}}{\bigoplus} H^1(K_{\infty,w}, E[p^m]) \to J[p^m] \to 0$$
This, in turn, induces another exact sequence
\begin{equation}\label{coker_psi_sequence1}
(\underset{w \in S_{\infty}}{\bigoplus} E(K_{\infty, w})/p^m)_{\Gamma_n} \to (\underset{w \in S_{\infty}}{\bigoplus} H^1(K_{\infty,w}, E[p^m]))_{\Gamma_n} \to J[p^m]_{\Gamma_n} \to 0
\end{equation}

In the proof of proposition \ref{surjectivity_prop} we showed that $J^*$ is a finitely generated torsion $\Lambda$-module. Hence it follows that $J[p^m]_{\Gamma_n}$ is finite. Also in lemma \ref{lambda_lemma} we showed that $(\underset{w \in S_{\infty}}{\bigoplus} E(K_{\infty, w})/p^m)_{\Gamma_n}$ is finite. Therefore it follows from the exact sequence (\ref{coker_psi_sequence1}) that $(\underset{w \in S_{\infty}}{\bigoplus} H^1(K_{\infty,w}, E[p^m]))_{\Gamma_n}$ is finite and hence $\coker \psi_{n,m}$ is finite.

It follows that $\ilim_{n,m} \coker \psi_{n,m} = \coker \psi$ where $\psi:=\ilim_{n,m} \psi_{n,m}$

Since the groups in the exact sequence (\ref{coker_psi_sequence1}) are finite, therefore by taking inverse limits the sequence remains exact
\begin{equation}\label{coker_psi_sequence2}
\ilim_{n,m} (\underset{w \in S_{\infty}}{\bigoplus} E(K_{\infty, w})/p^m)_{\Gamma_n} \to \ilim_{n,m} (\underset{w \in S_{\infty}}{\bigoplus} H^1(K_{\infty,w}, E[p^m]))_{\Gamma_n} \xrightarrow{\varphi} \ilim_{n,m} J[p^m]_{\Gamma_n} \to 0
\end{equation}
Now for any pair of maps of abelain groups $A \xrightarrow{\alpha} B \xrightarrow{\beta} C$ it is easy to prove from the snake lemma that we have an exact sequence
\begin{align*}
0 \to \ker(\alpha) &\to \ker(\beta \circ \alpha) \to \ker(\beta) \\
&\to \coker(\alpha) \to \coker(\beta \circ \alpha) \to \coker(\beta) \to 0
\end{align*}
Applying this to the maps
$$\ilim_{n,m} H^1(G_S(K_{\infty}), E[p^m])_{\Gamma_n} \xrightarrow{\psi} \ilim_{n,m} (\underset{w \in S_{\infty}}{\bigoplus} H^1(K_{\infty,w}, E[p^m]))_{\Gamma_n} \xrightarrow{\varphi} \ilim_{n,m} J[p^m]_{\Gamma_n}$$\\
and taking the sequence (\ref{coker_psi_sequence2}) into account we get an exact sequence
$$\ilim_{n,m} (\underset{w \in S_{\infty}}{\bigoplus} E(K_{\infty, w})/p^m)_{\Gamma_n} \to \coker(\psi) \to \coker (\varphi \circ \psi) \to 0$$
Since $J^*$ is a finitely generated $\Lambda$-module, it follows from lemma \ref{inverse_limit_lemma} that $\ilim_{n,m} J[p^m]_{\Gamma_n}$ is a finitely generated $\Zp$-module and so the same is true for $\coker(\varphi \circ \psi)$. Also in the proof of lemma \ref{lambda_lemma} we showed that $\ilim_{n,m} (\underset{w \in S_{\infty}}{\bigoplus} E(K_{\infty, w})/p^m)_{\Gamma_n}$ is finite. Therefore from the above exact sequence $\ilim_{n,m} \coker \psi_{n,m}=\coker \psi$ is a finitely generated $\Zp$-module. It follows that $\img \theta$ is finite since as we showed above it is a torsion $\Zp$-module.
It follows from this, the exact sequence (\ref{final_sequence}) and the observations noted after this sequence that we have an exact sequence
$$0 \to A \to \ilim_{n,m} \ker \psi'_{n,m} \to B \to 0$$
where $A \sim \dot{T}_{\lambda}(\Selinf(E/K_{\infty})^*)$ and $B \sim \dot{T}_{\mu}(\Selinf(E/K_{\infty})^*)$. So we get $\ilim_{n,m} \ker \psi'_{n,m} \sim \dot{T}_{\Lambda}(\Selinf(E/K_{\infty})^*)$. On the other hand $\ilim_{n,m} \ker \psi'_{n,m} = \ilim_{n,m} \Sha^2(G_S(K_n), E[p^m])$ which by the pairing (\ref{pairing2}) defined in the beginning of this section may be identified with the Pontryagin dual of $R_{p^{\infty}}(E/K_{\infty})$. This completes the proof of theorem \ref{main_theorem}.

\textbf{Acknowledgments} The author would like to thank K\k{e}stutis \v{C}esnavi\v{c}ius for many helpful correspondences about some of the arguments in the introduction. The author would also like to thank Robert Pollack for his helpful comments about this paper.

\end{document}